\documentclass[a4paper,11pt]{article}

  \usepackage{amsmath, latexsym, amsfonts, amssymb, amsthm, amscd,graphicx, color
}
  \usepackage{graphics,epsf,psfrag}
  \numberwithin{equation}{section}

  \newtheorem{theorem}{Theorem}[section]
  \newtheorem{lemma}[theorem]{Lemma}

  \newcommand{\ind}{\mathbf{1}}
  \newcommand{\E}{\mathbb{E}}
  \newcommand{\R}{\mathbb{R}}
  
  \newcommand{\N}{\mathbb{N}}
  
  \renewcommand{\hat}{\widehat}

  \newcommand{\cL}{{\ensuremath{\mathcal L}} }

  \newcommand{\bP}{{\ensuremath{\mathbb P}} }
  \newcommand{\bE}{{\ensuremath{\mathbb E}} }

  \DeclareMathSymbol{\leqslant}{\mathalpha}{AMSa}{"36} % nicer `smaller or equal'
  \DeclareMathSymbol{\geqslant}{\mathalpha}{AMSa}{"3E} % nicer `larger or equal'
  \DeclareMathSymbol{\eset}{\mathalpha}{AMSb}{"3F}     % nicer `emptyset'
  \renewcommand{\leq}{\;\leqslant\;}                   % redef. of < or =
  \renewcommand{\geq}{\;\geqslant\;}                   % redef. of > or =
  \newcommand{\gb}{\beta}
  \newcommand{\gep}{\varepsilon}       % \ge already exists...

  \newcommand{\go}{\omega}

  \newcommand{\gT}{\Theta}
  \def \ind {{\mathbf 1}}

  \makeatletter
  \def\captionfont@{\footnotesize}
  \def\captionheadfont@{\scshape}

  \long\def\@makecaption#1#2{%
    \vspace{2mm}
    \setbox\@tempboxa\vbox{\color@setgroup
      \advance\hsize-6pc\noindent
      \captionfont@\captionheadfont@#1\@xp\@ifnotempty\@xp
	  {\@cdr#2\@nil}{.\captionfont@\upshape\enspace#2}%
      \unskip\kern-6pc\par
      \global\setbox\@ne\lastbox\color@endgroup}%
    \ifhbox\@ne % the normal case
      \setbox\@ne\hbox{\unhbox\@ne\unskip\unskip\unpenalty\unkern}%
    \fi
    \ifdim\wd\@tempboxa=\z@ % this means caption will fit on one line
      \setbox\@ne\hbox to\columnwidth{\hss\kern-6pc\box\@ne\hss}%
    \else % tempboxa contained more than one line
      \setbox\@ne\vbox{\unvbox\@tempboxa\parskip\z@skip
	  \noindent\unhbox\@ne\advance\hsize-6pc\par}%
  \fi
    \ifnum\@tempcnta<64 % if the float IS a figure...
      \addvspace\abovecaptionskip
      \moveright 3pc\box\@ne
    \else % if the float IS NOT a figure...
      \moveright 3pc\box\@ne
      \nobreak
      \vskip\belowcaptionskip
    \fi
  \relax
  }
  \makeatother
  \def\writefig#1 #2 #3 {\rlap{\kern #1 truecm
  \raise #2 truecm \hbox{#3}}}
  
  \newcommand{\1}{{\rm 1}\kern-0.24em{\rm I}}

  \def\beq{\begin{equation}}
  \def\eeq{\end{equation}}

  \title{ Hierarchical pinning model: low disorder relevance in the $b=s$ case.}
  \author{Julien Sohier \footnote{ 
   Technische Universiteit Eindhoven, 
P.O. Box 513, 
5600 MB EINDHOVEN, 
\textit{j.sohier@tue.nl}   } \\}

\begin{document}

  \maketitle
  
   \begin{abstract}
  We consider a hierarchical pinning model introduced by B.Derrida, V.Hakim and J.Vannimenus which undergoes a localization/delocalization 
   phase transition. This model depends on two parameters $b$ and $s$. We show that in the particular case where $b=s$, the disorder is 
    weakly relevant, in the sense that at any given temperature, the quenched and the annealed critical points coincide. This is in contrast 
     with the case where $b \neq s$.
\end{abstract}

    

  \section{Model}

   The model we consider in this note can be interpreted either as an infinite dimensional system or as a pinning model 
    on a hierarchical lattice. For convenience, we will mainly choose to stick to the latter interpretation. We refer 
     to \cite{DHV} where this model was introduced and to \cite{MGa} for the latter interpretation. 
     
      We consider a family of iid random variables $(\go_n)_{ n \geq 0}$ with
 common law $\bP$ verifying
    \begin{equation}
      M(\gb) := \bE[\exp(\gb \go_1)] < \infty 
    \end{equation}
  for all positive $\gb$. We also consider positive integers $b,s \geq 2$.
    
  For $b \geq s$ and $(\gb,h) \in \R^+ \times \R$, we are interested in a system of recursions
 $(R_n^{(i)})_{ n \geq 0, i\geq 1}$ defined by: 
  \begin{equation} \label{Gen}
    R_{n+1}^{(i)} = \frac{ \prod_{j=1}^s R_n^{(si+j)} + b-1}{b}
  \end{equation}
  for $n, i \in \N$ where the random variables $(R_0^{(i)})_{i \geq 0}$ are given by 
  \begin{equation}
  R_0^{(i)} = \exp(\gb \omega_i +h - \log(M(\gb))).
  \end{equation}
  
  The annealed counterpart of \eqref{Gen} reads
    \begin{equation}\label{ann}
    r_0 =1 \hspace{0.05in} \text{and for} \hspace{0.05in} n \geq 0, r_{n+1} = \frac{r_n^s +b-1}{b}, 
    \end{equation}
   where for every $n \geq 0$, $r_{n} = \E\left[R_{n}^{(1)}\right]$.

  It is clear that for every $n \in \N$, the variables $\left(R_n^{(i)}\right)_{i \geq 0}$ are iid. The random variable 
$R_n^{(1)}$ can be interpreted as the partition function of the hierarchical pinning model with \textit{bond disorder}
 at step $n$. 
 
    \begin{figure}\label{fig1}
         \center{\scalebox{.35}{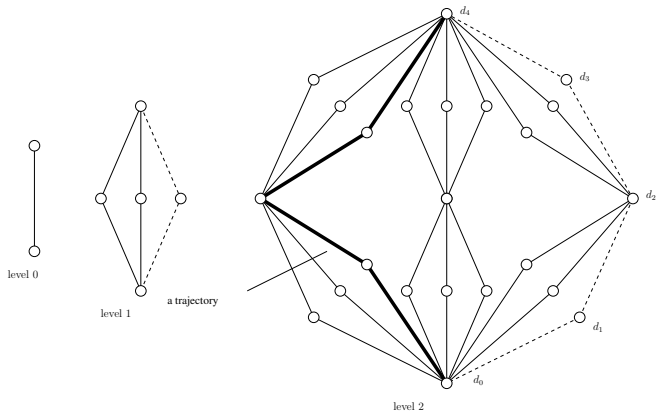}}
         \caption{ Given $b,s$ natural numbers ( here, $b=3,s=2$), we build a diamond iteratively: at each step, one replaces every bound by $b$ branches
          consisting of $s$ bounds. A trajectory of the process in a diamond lattice at level $n$ is a path connecting the two extreme poles; we 
           singled out one trajectory (which we denote by $\mathcal{L}$) using a thick line. At level $n$, each trajectory is made of $s^{n}$ bounds and there are $N_{n}$ trajectories,
            which satisfy $N_{0} = 1$, $N_{n+1} = b N_{n}^{s}$. A simple random walk at level $n$ is the uniform measure over the $N_{n}$ trajectories. 
            A special trajectory $\mathcal{L}$ with vertices labeled $d_{0}, \ldots, d_{s^{n}}$ is chosen and marked by a dashed line, we call it \textit{the wall boundary}.
           The pinning model is then built by rewarding or
            penalizing the trajectories according to their energy in the standard statistical mechanics
               fashion and the partition function of such a model is therefore given by $R_n^{(1)}$ in \eqref{part}. It is easy, and
 detailed in \cite{DHV} how to extract from this recursive construction  the recursion
\eqref{Gen}.
 }
       \end{figure}

  The model we consider may be viewed as a hierarchical version of pinning models, to which a big amount of litterature has 
  been devoted in the past couple of years (see \cite{GB}--\cite{GBlec} for recent reviews). 
  
 It has first been introduced in \cite{DHV}[Section 4.2], where the partition
function $R_{n}$ is defined as: 

 \begin{equation}\label{part}
  R_{n} = \mathbf{E}_{n} \left[ \exp\left( \sum_{i=1}^{s^{n}} (\gb \go_{i} - \log(M(\gb)) + h) \ind_{(S_{i-1} ,S_{i} )=(d_{i-1},d_{i} )} \right) \right],
 \end{equation} 
where $(S_{i})_{i = 0,\dots,s^{n}}$ is a simple random walk on a hierarchical diamond lattice with
growth parameter $b$ and $\mathbf{P}_{n}$ is the expectation with respect to the law of $S$. The labels $d_{0}, \dots, d_{s^{n}}$ are 
the labels of the vertices of a particular path $\cL$
that has been singled out and that we call \textit{the wall} in the sequel.  The construction of diamond lattices
and a graphical description of the model are detailed in Figure \ref{fig1} and its caption.

 This model, and the closely \textit{site disorder hierarchical pinning model}, has been extensively
 studied from a mathematical perspective in \cite{GLT}, in \cite{L} and in \cite{LT}; in particular it is well known that, for $ s > b$, the limit 
$ \lim_{N \to \infty} \frac{1}{s^N} \log R^{(1)}_N := F(\gb,h)$  exists almost surely and in $L^1(\bP)$, and we 
call this limit the free energy of the system. Note that $ F(\gb,\cdot)$ is convex and 
increasing. We define the critical point of the system:
  \begin{equation}
    h_c(\gb) := \inf \{ h \in \R \hspace{0.05in} \text{such that} \hspace{0.05in}  F(\gb,h) > 0 \}. 
  \end{equation}
  
  For an interpretation of $h_{c}(\gb)$ as the transition point between a delocalized and a localized
phase ($h < h_{c}(\gb)$ and  $h > h_{c}(\gb)$, respectively) we refer to \cite{DHV} and \cite{GLT}.

Our aim in this note is to deal with the (missing) case $b = s$; we show in particular that the
transition is of infinite order, in the sense that 
    close to its critical point, the annealed free energy of the system vanishes faster than any power of $h$. Then we 
     prove that in this particular case, at all temperatures, the critical values in the quenched and annealed models coincide, 
      which is in contrast with the case where $s > b$.
      
       \section{ The case $b=s$.}

    When $b=s$, our basic 
recursion \eqref{Gen} reads:
    \begin{equation} \label{DB}
  R_{n+1}^{(i)} = \frac{ \prod_{j=1}^s R_n^{(si+j)} + s -1}{s}. 
  \end{equation}

  The next result is the analogue of Theorem 1.1 of \cite{GLT}, whose proof is immediately adapted to the $b=s$ case.
  \begin{theorem} \label{DefEL}
      The limit $1/s^n \log\left(R_n^{(1)}\right)$ exists $\bP(d\go)$ almost surely and in $L^1(\bP)$, it is
 almost surely constant and non-negative. The function $ (\gb,h) \mapsto F(\gb,h+\log(M(\gb))$ is convex and
 $F(\gb,\cdot)$ is non-decreasing and convex. These properties are inherited from $F_N(\cdot,\cdot)$ defined by 
  \begin{equation}
    F_N(\gb,h) := \frac{1}{s^N} \bE\left[\log R^{(1)}_N\right].
  \end{equation}

  \end{theorem}
      
  The next result quantifies the fact that at criticality, the annealed free energy vanishes faster than any power of
 $h$. We stress that in the generic case where $s > b$, it is shown in \cite{GLT}[Theorem 1.2] that the critical 
  exponent of the system is equal to $\frac{\log(s) - \log(b)}{\log(s)}$, and hence Theorem \ref{EL} is the natural counterpart of this 
  result. 
    
    \begin{theorem}\label{EL}
      The function $ h \mapsto F(0,h) =: F(h)$ is real analytic on the whole real line. 
Moreover, $h_c(0) = 0$ and there exists $c > 1$ such that for all $ h \in (0,1)$, one has
  \begin{equation} \label{EQ}
    c^{-1} e^{ -\frac{ c  }{h}} \leq F(h) \leq c e^{ -\frac{ c^{-1}}{h}}.
  \end{equation}

  \end{theorem}
  
   A remarkable feature of these
     models  is the fact that in the particular case where $b=s$, in analogy with the disordered 
     pinning model with loop exponent one which has been treated in \cite{AZ} (and rededuced in \cite{CdH}[Corollary 1.7] via
      large deviation arguments), one can show 
      that the quenched critical point and the annealed one coincide for any given $\gb \geq 0$. We note that the equality of these critical points
       fails  at least at low temperature 
     as soon as $s > b$ and the $\go_{i}$'s are unbounded random variables (see \cite{GLT}[Corollary 4.2]).

  \begin{theorem}\label{HierBS} For 
 every positive $\gb$, one has the equality $h_c(\gb) = h_c(0)(=0)$.
  \end{theorem}
  
  \section{Proofs.}
  
  \begin{proof}[Proof of Theorem \ref{EL}]
  The proof is close to the one in \cite{GLT}[Theorem 1.2]. We first define the useful notation
  $ \hat{F}(\gb,\gep) := F(\gb,h)$ where $\gep := e^h -1$.
  
      Then we consider $a > 0$ verifying $\hat{F}(0,a) = 1$ (which is possible because of the convexity and monotonicity 
of $F(\gb,\cdot)$) and we define the sequence $(a_n)_{n \geq 0}$ by $a_0 =a$ and 
  \begin{equation}
    a_{n+1} =  (sa_n +1)^{1/s} -1.
  \end{equation}
   
    One readily realizes that, on the one hand $a_n \to 0$ as $n \to \infty$, and on the other hand 
$s \hat{F}(0,a_{n+1}) = F(a_{n})$; as a consequence, $\hat{F}(0,a_{n}) = 1/s^n$.

    We claim that $ a_n \sim \frac{2}{(s-1)n}$ as $n \to \infty$. For this, noting that the sequence $(a_n)_{n \geq 0}$ satisfies:
  \begin{equation}
    a_n - a_{n+1} = \frac{(s-1)a_{n+1}^2}{2} + o(a_{n+1}^2),
  \end{equation}
  we get that  $a_n \sim a_{n+1}$. Hence the sequence $v_n = 1/a_n$ verifies 
    \begin{equation}
    v_{n+1} - v_n = \frac{(s-1)a_{n+1}}{2a_n} + o(a_n) \to \frac{s-1}{2},
  \end{equation}
  and thus $v_n \sim (s-1)n/2$ as $n \to \infty$.

  From this we deduce that there exists $c > 0$ such that
  \begin{equation}
    e^{-{2 c \log(s) \over  (s-1)a_n}} \leq \hat{F}(0,a_{n}) \leq e^{-{2 c^{-1} \log(s) \over  (s-1)a_n}}. 
  \end{equation}

Finally we get that for all $n$ and all $\gep \in [a_n,a_{n+1}]$, one has:
  \begin{equation}
    \hat{F}(0,\gep) \geq  \hat{F}(0,a_{n+1}) \geq c^{-1} e^{-{2 c \log(s) \over  (s-1)\gep}},
    \end{equation}
  \begin{equation}
    \hat{F}(0,\gep) \leq  \hat{F}(0,a_{n}) \leq c e^{-{2 c^{-1} \log(s) \over  (s-1)\gep}},
  \end{equation}
 which ends the proof. 
  \end{proof} 
  
  Our proof of Theorem  \ref{HierBS}  is closely related to the one of \cite{AZ} in the non-hierarchical pinning case. Namely,
  we bound $F(\gb,h)$ from below by restricting the partition function to paths which visit the defect wall $\mathcal{L}$ 
exclusively in \textit{good $n$-diamonds}; roughly speaking, good diamonds are diamonds such that the partition function 
restricted to this diamond is large enough. It turns out that we can prove that the cardinality of these good $n$-diamonds has
positive density as $N \to \infty$.

  In what follows, we consider $\mathcal{A}_{n}$ an $n$--diamond, that is a diamond obtained from the construction
  described in Figure \ref{fig1} after $n$ steps.
   For $k < n$, we numerate the successive $s^{n-k}$ distinct $k$--diamonds
   $(\mathcal{A}^{1}_{k}, \ldots, \mathcal{A}^{s^{n-k}}_{k})$ of $\mathcal{A}$ intersecting $\mathcal{L}$
   in the natural way, and we refer to them as the \textit{subdiamonds of size $k$ of $\mathcal{A}_{n}$}.  

  Let $q_n$ be the probability that the random walk $S$ goes through a diamond of size $n$ without visiting $\mathcal{L}$, namely
  \begin{equation}
   q_{n} = \mathbf{P}_{n} \left[(S_{j-1},S_{j}) \neq (d_{i-1},d_{i}), i =1,\ldots,n  \right].
  \end{equation}  
   
 Let $p_{k,n}$ the probability that $S$ visits a subdiamond $\mathcal{A}^{j_{0}}_{k}$ (where $j_{0} \in [1,s^{n-k}]$)
  and does not visit $\mathcal{L}$ outside from $\mathcal{A}^{j_{0}}_{k}$; namely 
\begin{equation}
 p_{k,n} = \mathbf{P}_{n} \left[ (S_{(j_{0}-1) s^{k}},S_{(j_{0}-1)s^{k} +1}) \in \mathcal{A}^{j_{0}}_{k},
 (S_{j-1},S_{j}) \neq (d_{i-1},d_{i}), (S_{j-1},S_{j}) \notin \mathcal{A}^{j_{0}}_{k}  \right].
\end{equation} 
 
  It is easy to realise that this probability is independent of $j_{0} \in \left[1,s^{n-k}\right]$.

  To perform the usual energy gain/entropy cost for the proof of our main result, we need the following estimate on the quantity $p_{k,n}$:

  \begin{lemma}\label{prob}  As $n \to \infty$, we have the equivalence:
  \begin{equation}\label{amq}
    1 - q_n \sim { 2 \over (s-1)n}.
  \end{equation}
   
    Furthermore, as $k \to \infty$ (and for any $n > k$), we have:
  \begin{equation}\label{am}
  \log(p_{k,n}) \sim 2 \log(k/n).
  \end{equation} 
  \end{lemma}
  \begin{proof}
    Plainly, the sequence $(q_n)_{ n \geq 0}$ satisfies the following recursion:
  \begin{equation}
    q_{n+1} = \frac{q_n^s +s-1}{s}
  \end{equation}
  with $q_0 = 0$. It is clear that $q_n \to 1$ as $n \to \infty$. Note that $p_n := 1 - q_n$ satisfies 
  \begin{equation}
    p_{n+1} = p_n - (s-1)p_n^2/2 +o(p_n^2)
  \end{equation}
    with $p_0 =1.$ Using the same trick as for the proof of Theorem \ref{EL}, we directly get 
that $ p_n \sim { 2 \over (s-1)n}$. 

    Then we have the equality:
  \begin{equation}\label{rel}
    p_{k,n} = \prod_{j=k}^{n-1} q_j^{s-1}.
  \end{equation}
  
    Combining \eqref{rel} and \eqref{amq}, we deduce \eqref{am}.
  
  \end{proof}
  
%    \begin{figure}\label{fig1}
%          \center{\scalebox{.4}{\input{fignpas.pdf_t}}}
%     
%      
%    \end{figure}

   \begin{proof}[Proof of Theorem \ref{HierBS}]

  To prove Theorem \ref{HierBS}, it is enough to show that, given $\gb > 0$, for any $ h > h_{c}(0), F(\gb,h) > 0 $.

  We fix $k \geq k(h)$ large enough such that
   \begin{equation}\label{cok}
    \frac{\log(r_k)}{s^k} \geq F(h)/2.
   \end{equation}

  For $h > h_c(0)$, we then fix $n$ large (depending on $\gb,h$) such that:
  \begin{equation} \label{Cond}
   \log_{s} \left( \frac{8 s^{k} V\left(R_{k}^{(1)}\right)}{r_{k}^{2}} \right)  < n,
  \end{equation}
   where for a random variable $X$, we write $V(X)$ for the variance of $X$ and $\log_{s}(x) := \log(x)/\log(s)$.

  We fix $N > n$ and we consider $\mathcal{A}$ an $N$--diamond. 
  
  We say that the $i$--th $n$--subdiamond of $\mathcal{A}$ is \textit{good} (for $i \in \left[1;s^{N-n}\right]$) if 
  \begin{equation}
    \sum_{j= (i-1)s^{n-k}+1}^{is^{n-k}} R_k^{(j)} \geq s^{n-k} r_k/2
  \end{equation}
  and it is said \textit{bad} otherwise (where we recall that the partitions functions $\left(R_k^{(j)}\right)_{ 1 \leq j \leq s^{N-k}}$ of 
  the $k$--subdiamonds of $\mathcal{A}$  are defined in the 
   recursion \eqref{Gen}). We denote by $I_N := \{ i_1 < i_2 < \ldots < i_{|I_N|} \} $ the set of good
 $n$-diamonds. Making use of Chebychev's inequality, we get:
    \begin{equation}
    \bP [1 \notin I_N] \leq \frac{ 4 s^{k-n}  V\left(R_k^{(1)}\right)}{  \bE\left[ R_k^{(1)}\right]^2}.
    \end{equation}

  Thanks to  \eqref{Cond}, the last quantity is smaller than $1/2$, so that $ \bP[ 1 \in I_N] =: p_{good} \geq 1/2$. 
The sequence $(i_{j}-i_{j-1})_{j \geq 2}$ is iid with common law the law of a geometric random variable with parameter $p_{good}$,
 so that
  by the law of large numbers, as $N \to \infty$, the following convergence holds almost surely:
    \begin{equation}\label{usLG}
    \lim_{N \to \infty} \frac{|I_N|}{s^{N-n}} \to p_{good}.
    \end{equation}

      \begin{figure}\label{fig2}
         \center{\scalebox{.4}{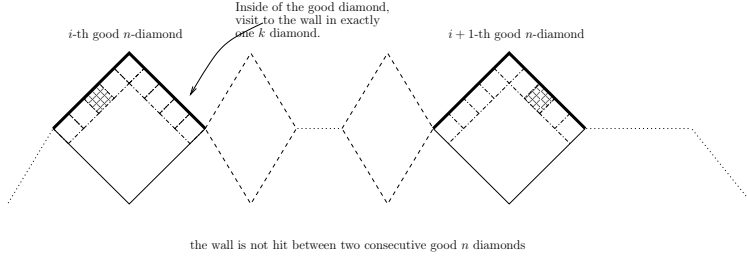}}
         \caption{ A graphical description of the set of trajectories $\gT_{N}^{I_N}$: it is the set of trajectories 
          hitting the wall $\cL$, which is depicted by a thick line, only within one $k$-subdiamond (here the filled $k$--diamonds) 
           of every good $n$--diamond. 
 }
       \end{figure}
      We denote by $\gT_{N}^{I_N}$ the set of trajectories of $S$ in $\mathcal{A}$ that do not intersect $\cL$ outside
      of good $n$-diamonds, and if they visit 
a good $n$-diamond, all the intersections to $\cL$ in this $n$-diamond are performed inside of the same sub $k$-diamond. In the following,
 we write $k_j$ for the location of the sub $k$-diamond in the $j$-th good $n$-diamond visited by $S$ ( we recall that
the $s^{n-k}$ sub $k$-diamonds are numbered in the natural way), and we denote by $l_{k_{j-1},k_j}$ the probability for the
 walk of leaving the $k_{j-1}$-th $k$-diamond of the $j-1$-th good $n$-diamond and visiting the $k_j$-th $k$-diamond of the
 $j$-th good $n$-diamond without visiting $\cL$ between these two $k$-diamonds. Finally, we write $J_j$ for the location of
 the  $j$-th good $n$-diamond. We refer to Figure \ref{fig2} for a graphical description of the set $\gT_{N}^{I_N}$.

  Denoting by $R_{N}( \Lambda)$ the partition function appearing in \eqref{part} for the 
  trajectories of $S$ restricted to a set of trajectories $\Lambda$, we have the inequality:
    \begin{equation}
  R_N( \gT_{N}^{I_N} ) \geq  \sum_{ k_1 \in [ i_1s^{n-k}; (i_1+1)s^{n-k} -1]} \ldots
  \sum_{ k_{|I_N|} \in [  i_{|I_N|}s^{n-k}; (i_{|I_N|}+1)s^{n-k} -1]} \prod_{j=1}^{|I_N|} R_{k}^{(k_j)} l_{k_{j-1},k_j}.
  \end{equation}

    The following equality holds:
  \begin{equation}
  \prod_{j=1}^{|I_N|} l_{k_{j-1},k_j} = \prod_{j=1}^{|I_N|} p_{k,n} q_n^{i_{j+1}-i_{j}+1}.
  \end{equation}
   
    This implies:
  \begin{equation}
  R_N( \gT_{N}^{I_N} ) \geq \prod_{j=1}^{|I_N|} \left[ \sum_{ k_j \in J_j} R_{k}^{(k_j)} \right] p_{k,n} q_n^{i_{j+1}-i_{j}+1}
  \end{equation}
    \begin{equation*}
  \geq \prod_{j=1}^{|I_N|} s^{n-k} {r_k \over 2} p_{k,n} q_n^{i_{j+1}-i_{j}+1}
  \end{equation*}
    because of the definition of a good $n$-diamond. From this last inequality, we deduce:
  \begin{equation}
  \frac{\log(R_N)}{s^N} \geq \frac{|I_N|}{s^{N}} \left( \log(s^{n-k}) -\log(2) +  \log(r_k)
  +  \log(p_{k,n}) \right) + \frac{ \sum_{j=1}^{|I_N|} (i_{j+1} - i_j +1) \log(q_n)}{s^N}. 
  \end{equation}
   
  Considering the limit when $N$ goes to infinity, making use of \eqref{usLG} and of the law of large numbers, we get the following inequality:
  \begin{equation}
    F(\gb,h) \geq \frac{p_{good}}{s^n} \left[ (n-k) \log(s)-\log(2) +  \log(r_k) + \log(p_{k,n}) + \log(q_n) \bE[i_1] \right].
  \end{equation}
  
    Considering $k$ large enough, and then $n$ large enough (in particular satisfying the condition \eqref{Cond}), we make use
    of Lemma \ref{prob} to deduce that there exists $c \in (0,1)$ which can be chosen arbitrarily close to $1$ such that:
  \begin{equation}
    F(\gb,h)  \geq \frac{p_{good}}{s^n} \left[ (n-k)\log(s) -\log(2) + \log(r_k) 
    +  2 c[\log(k) - \log(n)] -{2 \over c(s-1)} \frac{\bE[i_1]}{n} \right].
  \end{equation}
   
  Now we use the fact that $p_{good} > 1/2$ and \eqref{cok} to get that: 
  \begin{equation}
    F(\gb,h) \geq \frac{p_{good}}{s^n} \left[  (n-k)\log(s) -\log(2) + 2 c[\log(k) - \log(n)] + s^k F(h)/2 - {4 \over c(s-1)n}  \right],
  \end{equation}
    and this last quantity is (strictly) positive for $n$ large enough. 
  
     \end{proof}
     
      \textbf{Acknowledgments:} the  author thanks H. Lacoin for fruitful comments. Part of this work was supported by the
      NWO STAR grant: ``Metastable and cut-off behavior of stochastic processes''. 

   \bibliographystyle{plain}
 
 \bibliography{JSbs}

  \end{document}